\documentclass[12pt]{article}
\usepackage[utf8]{inputenc}
\usepackage{stix}
\usepackage{amssymb,amsmath,amsfonts,mathrsfs,amsthm} 
\usepackage[colorinlistoftodos]{todonotes}
\usepackage[allcolors=blue]{hyperref}
\usepackage{empheq}
\usepackage{bm}
\newtheorem{theorem}{Theorem}
\newtheorem{proposition}[theorem]{Proposition}
\newtheorem{lemma}[theorem]{Lemma}

\newtheorem{remark}[theorem]{Remark}

\theoremstyle{definition}
\newcommand{\R}{\mathbb{R}}
\newcommand{\Q}{\mathbb{Q}}
\newcommand{\Sf}{\mathbb{S}}

\newcommand{\Hy}{\mathbb{H}}

\newcommand{\grad}{\mbox{grad\,}}

\def\<{{\langle}}
\def\>{{\rangle}}

\def\id{\mbox{id}}

\def\be{\begin{equation} }
\def\ee{\end{equation} }

\def\proof{\noindent{\it Proof:  }}
\def\qed{\ifhmode\unskip\nobreak\fi\ifmmode\ifinner
\else\hskip5 pt \fi\fi\hbox{\hskip5 pt \vrule width4 pt
height6 pt  depth1.5 pt \hskip 1pt }}
\makeatletter
\newcommand{\subjclass}[2][]{\let\@oldtitle\@title
\gdef\@title{\@oldtitle\footnotetext{#1 
\emph{Mathematics Subject Classification:} #2}}}
\newcommand{\keywords}[1]{\let\@@oldtitle\@title
\gdef\@title{\@@oldtitle\footnotetext
{\emph{Key words and phrases.} #1.}}}
\makeatother

\title{}
\author{A. N. S. Carvalho\footnote{Corresponding author}  and R. Tojeiro
      \footnote{The first author was supported by  CAPES-PROEX grant 12498478/D. The second author was partially supported by  Fapesp grant 2022/16097-2 and CNPq grant 307016/2021-8.\\
      Data availability statement: Not applicable. \\
      \\
Emails: \texttt{arnandonelio@usp.br}, \texttt{tojeiro@icmc.usp.br}.
}}
\date{}

\begin{document}
\title{Constant curvature hypersurfaces of cylinders over space forms}
\maketitle

\begin{abstract}{
We classify the hypersurfaces of dimension $n\geq 3$ with constant sectional curvature of 
$\mathbb{R}^k\times \mathbb{S}^{n-k+1}$ and $\mathbb{R}^k\times \mathbb{H}^{n-k+1}$,  $2\leq k\leq n-1$.}   
\end{abstract}

\vspace{2ex}


\section{Introduction}

The problem of classifying (locally or globally)  submanifolds with constant sectional curvature
of different ambient spaces is an important topic in Submanifold theory. When the ambient space is a space form, the problem has a long history, with outstanding contributions by E. Cartan \cite{Ca} in the higher dimensional case. In the particular case of hypersurfaces with dimension $n\geq 3$, a complete classification is well known (see, eg., \cite{Tojeiro}). 

  The case in which the ambient space is a product of space forms has recently attracted the attention of several geometers. 
  In particular, hypersurfaces with constant curvature of $\mathbb{S}^2\times \mathbb{S}^2$ were classified in \cite{Luc1}.
  The classification in the case of $\mathbb{H}^2\times \mathbb{H}^2$ as the ambient space was recently achieved by the same authors in \cite{Luc2}. 
  
  On the other hand, hypersurfaces of $\mathbb{R}\times \mathbb{S}^n$ and $\mathbb{R}\times \mathbb{H}^n$ with dimension $n\geq 3$ and constant sectional curvature were classified in \cite{ManfioTojeiro}, after a global classification was obtained for the case $n=2$ in \cite{AEG}.    
  
  In this article, we extend the classification in \cite{ManfioTojeiro} to the case in which the ambient space is $\mathbb{R}^k\times \mathbb{S}^{n-k+1}$ or $\mathbb{R}^k\times \mathbb{H}^{n-k+1}$, with  $2\leq k\leq n-1$. Together with recent results in \cite{KNT}, this completes the classification of all hypersurfaces with dimension $n\geq 3$ and constant sectional curvature of a product of two space forms.

 The paper is organized as follows. In the first section, we summarize from \cite{LTV} the basic equations of a hypersurface of a product of two space forms. The second section is devoted to the proof of a nonexistence result for hypersurfaces with nonzero constant sectional curvature in $\mathbb{R}^k\times \mathbb{S}^{n-k+1}$ and $\mathbb{R}^k\times \mathbb{H}^{n-k+1}$,
  $2\leq k\leq n-1$ (see Theorem \ref{nonexistence}). In the last section, we provide a local classification of the flat hypersurfaces of $\mathbb{R}^k\times \mathbb{S}^{n-k+1}$ and $\mathbb{R}^k\times \mathbb{H}^{n-k+1}$,   $2\leq k\leq n-1$ (see Theorems~\ref{thm:flat_sphere} and \ref{thm: flatH}). In particular, it is shown that flat hypersurfaces only exist in $\mathbb{R}^k\times \mathbb{S}^{n-k+1}$,  $2\leq k\leq n-1$,  if $k\in \{n-1, n-2\}$.

 \section{Preliminaries}

The basic equations of a submanifold with arbitrary dimension and codimension of a product of two space forms were established independently in \cite{Kow}  and \cite{LTV}. We summarize below those equations in the special case of hypersurfaces,  in which we are interested in this article, following the approach in \cite{LTV}.

  Given an isometric immersion  $f\colon M^n \to  \mathbb{Q}_{c_1}^{k}\times \mathbb{Q}_{c_2}^{n-k+1}, 2 \leq k \leq n-1,$ into a product of two space forms of constant sectional curvatures $c_1$ and $c_2$, 
   let  $N$ be a (local) unit normal vector field along $f$ and let $A$ be the shape operator of $f$ with respect to $N$. Let $\pi_1$ (respectively, $\pi_2$) denote both the projection of $\mathbb{Q}_{c_1}^{k}\times \mathbb{Q}_{c_2}^{n-k+1}$ onto $\mathbb{Q}_{c_1}^k$ (respectively, $\mathbb{Q}_{c_2}^{n-k+1}$) and its derivative, and let $t=t_f\in C^{\infty}(M)$,
   $\xi=\xi_f\in \mathfrak{X}(M)$, and $R=R_f\in \Gamma(T^*M\otimes TM)$ be defined by
\begin{equation} \label{t}
\pi_2N=f_*\xi+tN
\end{equation}
and 
\begin{equation}\label{eq:L} \pi_2 f_*X=f_*RX+\langle X, \xi\rangle N
\end{equation}
for all  $ X \in \mathfrak{X}(M)$. 
\begin{remark}\label{cod} \emph{For isometric immersions  $f\colon M^n \to  \mathbb{Q}_{c_1}^{n_1}\times \mathbb{Q}_{c_2}^{n_2}$ with arbitrary codimension, equations \eqref{t} and \eqref{eq:L} are expressed, for all $\eta\in \Gamma(N_fM)$ and $X\in \mathfrak{X}(M)$, as
$$\pi_2\eta=f_*S^t\eta+T\eta$$
and
$$\pi_2 f_*X=f_*RX+SX,$$
in terms of $R$ and tensors $S\in \Gamma(T^*M\otimes N_fM)$ and $T\in \Gamma((N_fM)^*\otimes N_fM)$, where $S^t$ denotes the transpose of $S$. In the case of hypersurfaces, the tensor $S$ can be seen as a one-form on $M$, and can be expressed in terms of the vector field $\xi$ and a unit normal vector field $N$ by $SX=\<\xi, X\>N$ for all $X\in \mathfrak{X}(M)$. In particular, vanishing of $S$ is equivalent to the vanishing of $\xi$. In turn, the tensor $T$ is determined by the single function $t\in C^\infty(M)$.}
\end{remark}

From
$\pi_2^2=\pi_2$
it follows that
\begin{equation}\label{eq:pi1}
R(I-R)X=\langle X, \xi\rangle \xi,\,\,\,\,\,((1-t)I-R)\xi=0\,\,\,\,\,\mbox{and}\,\,\,\,\, t(1-t)=\|\xi\|^2
\end{equation}
for all $X\in \mathfrak{X}(M)$. The definition of
$R$ implies that it is a symmetric endomorphism. By the first equation in \eqref{eq:pi1}, if $\xi$ vanishes at  $x\in M^n$, then $T_xM$ decomposes orthogonally as $T_xM=\ker R \operp \ker (I-R)$, hence $0$ and $1$ are the only eigenvalues of $R$. Otherwise, $T_xM$ decomposes orthogonally as
\begin{equation}\label{decomp}
T_xM=\ker R \operp \ker (I-R)\operp  \mbox{span} \, \{\xi\},
\end{equation}
in which case, in view of the second equation in \eqref{eq:pi1},  the eigenvalues of $R$ are $0$, $1$, and $r=1-t$, with corresponding eigenspaces $\ker R$, $\ker (I-R)$, and $ \mbox{span} \, \{\xi\}$. Moreover, in Lemma 3.2 of \cite{BrunoandRuy} it was shown that the subspaces $\ker R$ and $\ker (I-R)$ give rise to smooth subbundles of $TM$ in open subsets where they have constant dimension. 

  It follows from \eqref{eq:L} and \eqref{decomp}  that  $X\in \ker R$ if and only if $\pi_{2}f_{*}X=0$. Similarly,  $X\in \ker (I-R)$ if and only if $\pi_{1}f_{*}X=0$. Thus,
  $$f_*\ker R(x)= f_*T_xM\cap T_{\pi_1(f(x))}\mathbb{Q}_{c_1}^{k} \,\,\mbox{and}\,\, f_*\ker (I- R)(x)= f_*T_xM\cap T_{\pi_2(f(x))}\mathbb{Q}_{c_2}^{n-k+1}$$ 
    for every $x\in M^n$. Since $f_*T_xM$ has codimension one in $T_{f(x)}(\mathbb{Q}_{c_1}^{k}\times \mathbb{Q}_{c_2}^{n-k+1})=T_{\pi_1(f(x))}\mathbb{Q}_{c_1}^{k} \oplus  T_{\pi_2(f(x))}\mathbb{Q}_{c_2}^{n-k+1}$, this implies the following observation.

   \begin{lemma}\label{lemma for dimension of kernel of Rb}
    Let $f\colon M^{n}\rightarrow \mathbb{Q}_{c_1}^{k}\times \mathbb{Q}_{c_2}^{n-k+1}$ be an isometric immersion. Then 
    \begin{equation} \label{ineq}
    k-1\leq \dim \ker R\leq k\quad \mbox{and}\quad n-k\leq \dim \ker (I-R)\leq n-k+ 1.
    \end{equation}
\end{lemma}

We say that a hypersurface 
$
f\colon M^{n}\rightarrow \mathbb{Q}_{c_1}^{k}\times \mathbb{Q}_{c_2}^{n-k+1}
$ 
\emph{splits} if $M^n$ is (isometric to) a Riemannian product 
$
M_1^{k-1}\times \mathbb{Q}_{c_2}^{n-k+1}
$ 
or 
$
\mathbb{Q}_{c_1}^{k}\times M_2^{n-k}
$ 
and either 
$
f=f_{1}\times \operatorname{id}_{2}\colon M_1^{k-1}\times \mathbb{Q}_{c_2}^{n-k+1}\to  \mathbb{Q}_{c_1}^{k}\times \mathbb{Q}_{c_2}^{n-k+1} 
$ 
or 
$
f=\operatorname{id}_{1}\times f_{2}\colon \mathbb{Q}_{c_1}^{k}\times M_2^{n-k}\rightarrow \mathbb{Q}_{c_1}^{k}\times \mathbb{Q}_{c_2}^{n-k+1},
$ 
respectively, where 
$
f_{1}\colon M_1^{k-1}\rightarrow \mathbb{Q}_{c_1}^{k}
$ 
and 
$
f_{2}\colon M_2^{n-k}\rightarrow  \mathbb{Q}_{c_2}^{n-k+1}
$ 
are hypersurfaces and 
$\operatorname{id}_{1}$ and $\operatorname{id}_{2}$ are the respective identity maps. \vspace{1ex}

 \begin{proposition}\label{product locally splits}
     A hypersurface $f\colon M^{n}\rightarrow \mathbb{Q}_{c_1}^{k}\times \mathbb{Q}_{c_2}^{n-k+1}$, $2\leq k\leq n-1$,  splits locally if and only if the vector field $\xi$ vanishes. In particular,  if either $\dim\ker R=~k$ or $\dim \ker (I-R)=n-k+1$, then $f$ splits locally.
 \end{proposition}
 \proof Notice that, by \eqref{ineq}, one can have neither $R=0$ nor $R=I$ if $2\leq k\leq~n~-~1$. Taking into account Remark \ref{cod}, the first statement is then a consequence of Proposition $3.3$ in \cite{BrunoandRuy}. If  $\dim\ker R=k$ , since $\dim\ker(I-R)\geq n-k$ by \eqref{ineq},  then  $TM=\ker R\operp \ker (I-R)$, that is, $\xi$ vanishes identically, and hence $f$ splits locally. Similarly if $\dim \ker (I-R)=n-k+1$.\vspace{1ex}\qed

Computing the tangent and normal components in
$\nabla\pi_2=\pi_2\nabla$, applied to both tangent and normal vectors, and using the Gauss and Weingarten formulae, we obtain
\begin{equation}\label{derR}
(\nabla_XR)Y= \langle Y, \xi\rangle AX+\langle AX,Y\rangle \xi,
\end{equation}
\begin{equation}\label{derS2}
\nabla_X\xi=(tI-R)AX,
\end{equation}
for all $X, Y\in\mathfrak{X}(M)$, 
and 
\begin{equation}\label{dert2}
\grad t=-2A\xi.
\end{equation}

The Gauss and Codazzi equations of $f$ are
$$
R(X,Y)=c_1(X\wedge Y-X\wedge RY-RX\wedge Y) + (c_1+ c_2)RX\wedge RY +AX\wedge AY
$$
and 
\begin{equation}\label{codazzi}
(\nabla_XA)Y-(\nabla_YA)X=
(c_1(I-R)-c_2R)(X\wedge Y)\xi
\end{equation}
for all $X,Y\in\mathfrak{X}(M)$. \vspace{2ex}

Let $\R^N_\rho$ denote an Euclidean space with a standard flat metric of index $\rho$, and consider the standard inclusion  
$$
j\colon\,\Q_{c_1}^{k}\times \Q_{c_2}^{n-k+1}\to\R_{\sigma(c_1)}^{N_1}\times \R_{\sigma(c_2)}^{N_2}=\R_{\mu}^{N_1+N_2}, \,\,\,\mu = \sigma(c_1) + \sigma(c_2),
$$
where  $ \sigma(c)=0 $ if $ c \geq 0$,  $ \sigma(c)=1 $ if $c<0$,  
$$
N_1 = \begin{cases}
    k +1, \text{ if } c_1 \neq 0, \\ 
    k , \text{ if } c_1 = 0,
\end{cases}  
\quad \text{ and } \quad 
N_2 = \begin{cases}
    n-k +2, \text{ if } c_2 \neq 0, \\
    n-k+1 , \text{ if } c_2 = 0.
\end{cases} 
$$
Given a hypersurface $f\colon\,M^n\to \Q_{c_1}^{k}\times \Q_{c_2}^{n-k+1}$, denote $\tilde f=j\circ f$.
Decompose the position vector field $\tilde{f}$ as 
$
\tilde{f}=\tilde f_1+\tilde f_2,
$ 
where $\tilde f_1$ and $\tilde f_2$ are the corresponding components in $ \mathbb{R}_{\sigma(c_1)}^{N_1}$ and $ \mathbb{R}_{\sigma(c_2)}^{N_2}$, respectively.
If both $c_1$ and $c_2$ are nonzero, write $c_i=\epsilon_i/r_i^2$, $1\leq i\leq 2$, where $\epsilon_i$ is either $1$ or $-1$, depending on whether $c_i>0$ or $c_i<0$, respectively. 
Then the second fundamental form of $\tilde{f}$ is 
$$
    \alpha^{\tilde f}(X,Y)=\<AX,Y\>j_*N+\frac{\epsilon_1}{r_1}\<(I-R)X,Y\>)\tilde f_1+\frac{\epsilon_2}{r_2}\<RX,Y\>\tilde f_2
$$
for all $X,Y \in \mathfrak{X}(M)$. If $0=c_1\neq c_2$, then  
\begin{equation}\label{alpha2}
    \alpha^{\tilde f}(X,Y)=\<AX,Y\>j_*N+\frac{\epsilon_2}{r_2}\<RX,Y\>\tilde f_2
\end{equation} 
for all $X,Y \in \mathfrak{X}(M)$.

 \section{A nonexistence theorem in the nonflat case}
 
 Our first result deals with hypersurfaces with nonzero constant sectional curvature of  $ \mathbb{R}^2\times \mathbb{Q}^2_\epsilon$,  $\epsilon\in \{-1, 1\}$.

\begin{proposition} \label{prop:n=3} There does not exist any hypersurface with nonzero constant sectional curvature in  $ \mathbb{R}^2\times \mathbb{Q}^2_\epsilon$,  $\epsilon\in \{-1, 1\}$.
\end{proposition}

For the proof, we shall use the following fact, whose proof can be found in \cite{ManfioTojeiro}.

\begin{proposition}\label{Proposition ManfioTojeiro}
    Any isometric immersion $g: M_{c}^{n} \rightarrow \mathbb{R}_\mu^{n+2}$, $\mu\in \{0,1\}$, of a Riemannian manifold with dimension $n \geq 3$ and constant sectional curvature $c \neq 0$ has a flat normal bundle.
\end{proposition}
\noindent \emph{Proof of Proposition \ref{prop:n=3}:}
Assume that 
$ 
f\colon M_c^{3}\rightarrow \mathbb{R}^2\times \mathbb{Q}_{\epsilon}^{2} 
$, $\epsilon\in \{-1, 1\}$, 
is a hypersurface with constant sectional curvature $c\neq0$, and let $ A $ be the shape operator of $ f $ with respect to a (local) unit normal vector field $ N $. 
Let $\tilde{f} = j \circ f$ be the composition of $f$ with the  inclusion $j\colon \mathbb{R}^2\times \mathbb{Q}_{\epsilon}^{2}\to\mathbb{R}^2\times \mathbb{R}^3_{\sigma(\epsilon)}=\mathbb{R}^5_{\sigma(\epsilon)}$.
By Proposition~\ref{Proposition ManfioTojeiro}, $\tilde{f}$ has a flat normal bundle, which is equivalent to saying that $R$ and $A$ commute.
Since $c\neq 0$, $f$ cannot split in any open subset of $M_c^{3}$. Thus $ \xi \neq 0 $  in an open and dense subset ${V}\subset M_c^{3}$. From now on, we work on ${V}$.
Therefore, $R$ has eigenvalues $0$, $r\in (0,1)$, and $1$, all of which with multiplicity one.  
Let  $ \{U, \| \xi \|^{-1}\xi, W\} $ be an orthonormal frame that diagonalizes simultaneously $ A $ and $ R $, with 
$$
\begin{array}{l}
AU = \lambda U, \,\,\,\, A\xi = \mu \xi, \,\,\text{ and } \,\,AW =\sigma W,\vspace{1ex}\\
RU=0, \,\,\,\, R\xi = r\xi, \,\,\text{ and } \,\,RW = W.
\end{array}
$$
The Gauss equation of $ f $ is  
\begin{equation}\label{lambdas}
c(X \wedge Y)Z =  (AX\wedge AY)Z + \epsilon(RX\wedge RY)Z \nonumber 
\end{equation}
for all $X, Y, Z\in \mathfrak{X}({V})$.
Thus
$$ \lambda\mu=\lambda\sigma=\mu\sigma + \epsilon \,r=c.$$
From $\lambda\mu=c$ or $\lambda \sigma=c$ it follows that  $\lambda \neq 0$. Then $\lambda\mu=\lambda \sigma$ implies that $\mu = \sigma$.
We can thus rewrite the preceding equations as 
\begin{equation}\label{relations betwwen lambda and mu}
        \lambda\mu=
           c= \mu^2 + \epsilon \, r.    
\end{equation}
Note that \eqref{relations betwwen lambda and mu}  implies that $\mu \ne \lambda$, for $r \in (0,1).$

Since $\xi$ is an eigenvector of $A$, it follows from \eqref{dert2} that $W(r)=\<\grad r, W\>=0$. From the second equality in \eqref{relations betwwen lambda and mu} it follows that 
\begin{equation}\label{derivada de lambda e mu}
    W(\mu)=0.
\end{equation} 
Eq.  \eqref{derR} gives
$
R \nabla_\xi W= \nabla_\xi W, 
$
and since $ \ker (I-R) = \mbox{span}\, \{W\}$ and $W$ is a unit vector field, then
\begin{equation}\label{derivada de W na direção xi}
\nabla_\xi W =  0.
\end{equation}
 Eq. \eqref{derS2} yields
\begin{equation}\label{deriavada de xi na direção de W}
\nabla_W \xi = (tI -R)AW=-\mu rW. \\
\end{equation}
Using \eqref{derivada de lambda e mu}, \eqref{derivada de W na direção xi}, and \eqref{deriavada de xi na direção de W}, 
we obtain 
\begin{equation} 
\begin{split}
(\nabla_\xi A)W - (\nabla_WA)\xi &= \xi(\mu)W + \mu\nabla_\xi W - A\nabla_\xi W - W(\mu)\xi - \mu\nabla_W \xi +A\nabla_W \xi \nonumber \\ 
&=  \xi(\mu)W + \mu^2rW - \mu^2rW \\
&=\xi(\mu)W.
\end{split}
\end{equation} 
On the other hand, applying \eqref{codazzi} to $W$ and $\xi$ gives
\begin{equation} 
(\nabla_\xi A)W - (\nabla_WA)\xi = -\epsilon (\<W, \xi \> R\xi- \< \xi, \xi \> RW ) =\epsilon\|\xi\|^2W. \nonumber
\end{equation}
Therefore,
\begin{equation}\label{derivada de mu na direção xi}
    \xi(\mu) = \epsilon \|\xi\|^2.
\end{equation}
By \eqref{dert2} we have
\begin{equation}\label{derivada de r na direção xi}
    \xi(r) = \< \grad r, \xi \>= - \< \grad t, \xi\> = \<2A\xi, \xi\> = 2\mu\|\xi\|^2.
\end{equation}
Differentiating the second equation in \eqref{relations betwwen lambda and mu} in the direction of $\xi$ and using \eqref{derivada de mu na direção xi}  and \eqref{derivada de r na direção xi} yield
\begin{equation}
0 = 2\mu \xi(\mu) + \epsilon\xi(r) = 2\epsilon\mu \|\xi\|^2 + 2\epsilon\mu\|\xi\|^2 = 4\epsilon\mu\|\xi\|^2, \nonumber
\end{equation}
which is a contradiction, for $\mu \neq 0$ and $\|\xi\| \neq 0.$
\vspace{1ex}\qed

Before proving the nonexistence of hypersurfaces with nonzero constant sectional curvature of $\mathbb{R}^k \times \mathbb{Q}^{n-k+1}_\epsilon$,  $\epsilon\in \{-1, 1\}$,  for \(n \geq 4\) and \(2 \leq k \leq n-1\), we recall some facts on flat bilinear forms and isometric immersions between space forms.

A symmetric bilinear form 
$\beta \colon V \times V \to W, $ where $V$ and $W$ are finite-dimensional vector spaces, is \emph{flat} with respect to an inner product 
 $
 \langle \cdot , \cdot \rangle \colon 
 W \times W \to \mathbb{R}
 $
 if 
 $$
 \langle \beta(X,Y), \beta (Z,T) \rangle - \langle \beta(X,T), \beta (Z,Y) \rangle = 0
 $$
 for all $X, Y, Z, T \in V$. 
The \emph{nullity subspace} of $\beta$ is defined as
$$
\mathcal{N}(\beta) =\{ X \in V \, \,\colon \beta(X,Y) = 0 \text{ for all } Y \in V \},
$$
and its \emph{image subspace} by 
$$
\mathcal{S}(\beta) = \mbox{span}\, \{   \beta(X,Y)  : X, Y \in V\}.
$$
 If the inner product  $\langle \cdot , \cdot \rangle $ on $W$ is either positive definite or Lorentzian, then the inequality 
 \begin{equation}\label{fbf}
\dim \mathcal{N}(\beta)\geq \dim V- \dim \mathcal{S}(\beta)
\end{equation}
holds, provided that $ \mathcal{S}(\beta) $ is nondegenerate if $\langle \cdot , \cdot \rangle $ is Lorentzian (see Lemma 4.10 and Lemma 4.14 in \cite{Tojeiro}).

Item $(i)$ of the following theorem is due to Cartan \cite{Ca}, and item $(ii)$ to O'Neill \cite{O} (see also \cite{Mo} and Theorem $5.1$ in \cite{Tojeiro}).

\begin{theorem}\label{th: Cartan}
Let $f\colon M_c^n\to \mathbb{Q}^{n+p}_{\tilde c}$ be an isometric immersion of a Riemannian manifold with constant sectional curvature $c$. Then the following assertions hold:
\begin{itemize}
\item[(i)] If $ c < \tilde c,$ then $p \geq n-1.$
\item[(ii)] If $c > \tilde c$ and $ p \leq n-2,$ then for any $ x \in M^n$ there exist a unit vector $ \zeta \in N_fM(x)$ and a flat bilinear form 
$$
    \gamma_x\colon T_xM \times T_xM \to \{ \zeta \}^{\perp} \subset N_fM(x)
$$
such that the second fundamental form of $f$ at $x$ is given by 
$$
\alpha^f(x) = \gamma_x + \sqrt{c-\tilde c\,} \langle \, , \rangle \zeta.
$$
\end{itemize}
\end{theorem}

We are now in a position to prove the following nonexistence result.

\begin{proposition} \label{prop:n > 3}  There does not exist any hypersurface with dimension $n\geq 4$ and nonzero constant sectional curvature in  $\mathbb{R}^k\times\mathbb{Q}^{n-k+1}_\epsilon$, $\epsilon\in \{-1, 1\}$, if  $ 2\leq k \leq n-1$. 
\end{proposition}

\proof
Assume that 
$f\colon M_c^n\to \mathbb{R}^k\times \mathbb{Q}^{n-k+1}_\epsilon$, $ n \geq 4$ and $ 2\leq k \leq n-1$,
is an isometric immersion of a Riemannian manifold with constant sectional curvature $c \neq 0,$ let $N$ be a (local) unit normal vector field to $f$, and let $A$ be the shape operator of $f$ with respect to $N$. 
We argue separately for the cases \(\epsilon = 1\) and  \(\epsilon = -1\).\vspace{1ex}\\
{\bf Case $\epsilon=1$}:
Let 
$
j \colon \mathbb{R}^k\times \mathbb{S}^{n-k+1}\to \mathbb{R}^k\times\mathbb{R}^{n-k+2}=\mathbb{R}^{n+2}
$ 
be the inclusion, and let 
$
\tilde f=j \circ f \colon  M^n \to \mathbb{R}^{n+2} 
$ 
be its composition with $f$.
It follows from item $(i)$ of Theorem~\ref{th: Cartan} that $c> 0$. By item $(ii)$,  for any $x \in M$ there exists an orthonormal basis $\{\eta, \zeta\}$ of  $ N_{\tilde f}M(x)$
such that $
A^{\tilde f}_\zeta=\sqrt{c}I$
 and 
$\mbox{rank}\,  A^{\tilde f}_\eta\leq 1$. 
Let  $\theta \in \mathbb{R}$  be such that
$$
j_*N=\cos \theta \eta +\sin \theta \zeta\,\,\,\,\mbox{and}\,\,\,\, \tilde{\pi}_2\circ \tilde{f}=-\sin\theta \eta+\cos\theta \zeta,
$$
where $\tilde{\pi}_2\colon \mathbb{R}^{n+2}\to \mathbb{R}^{n-k+2}$ is the orthogonal projection.
Then
\begin{equation}\label{A A R}
\sqrt{c} \, I= \sin\theta\, A+\cos \theta \, R \,\,\,\,\,\mbox{and}\,\,\,\,\,
A^{\tilde f}_\eta= \cos\theta\, A - \sin\theta \,R.
\end{equation}
Note that $\sin \theta$  cannot be equal to $0$, for otherwise the first of the preceding equations would contradict the fact that $\dim \ker R \geq k-1 \geq 1$.
Thus, 
\begin{equation}\label{eq:af}
A|_{\ker R}=\frac{\sqrt{c}}{\sin \theta}I, \,\,\,A|_{\ker (I-R)}=\frac{\sqrt{c}-\cos \theta}{\sin \theta}I\,\,\,\mbox{and}\,\,\,A\xi=\frac{\sqrt{c}-r\cos \theta}{\sin \theta}\xi,
\end{equation}
where $r$ is the eigenvalue of $R$ in $(0,1)$. Using \eqref{A A R} and \eqref{eq:af} we obtain 
$$
    A^{\tilde f}_\eta |_{\ker R} = \sqrt{c} \dfrac{\cos \theta}{ \sin \theta} I, \, \,\,A^{\tilde f}_\eta |_{\ker(I-R)} =  \dfrac{\sqrt{c}\cos \theta -1}{ \sin \theta} I, \text{ and } \,\,A^{\tilde f}_\eta \xi =  \dfrac{\sqrt{c}\cos \theta -r}{ \sin \theta} \xi. 
$$
Since both $\ker R$ and $\ker(I-R)$ are nontrivial,  the preceding equations imply that $\mbox{rank}\,  A^{\tilde f}_\eta \geq 2,$  a contradiction that  completes the proof in this case.\vspace{1ex}\\
{\bf Case $\epsilon=-1$}: Let 
$
j \colon \mathbb{R}^k\times \mathbb{H}^{n-k+1}\to \mathbb{R}^k\times\mathbb{R}_1^{n-k+2}=\mathbb{R}_1^{n+2}
$ 
be the canonical embedding, and let 
$
\tilde f=j \circ f \colon  M^n \to \mathbb{L}^{n+2} 
$ 
be its composition with $f$. Let $i$ be an umbilical inclusion of $\mathbb{R}_1^{n+2}$ into either $\mathbb{H}_c^{n+2,1}$ or $\mathbb{S}_c^{n+1, 2}$, depending on whether $c < 0$ or $c>0$, respectively. Here, $\mathbb{Q}_c^{N,\mu}$ denotes an $N$-dimensional pseudo-Riemannian manifold with a metric of index $\mu$ and constant curvature $c$. Then the second fundamental form of $F=i\circ \tilde f$ at $x$ is 
\begin{equation}\label{eq13}
\alpha^F(X,Y) = i_*\alpha^{\tilde{f}}(X,Y) + \sqrt{|c|} \<X,Y\>\eta
\end{equation}
for all $ X,Y \in T_xM$,
where $\eta$ is one of the two unit normal vectors to $i$ at $\tilde f(x)$.  Replacing the inner product $\< \,, \>$ in $N_{\tilde f}M(x)$ by $-\<\, , \>$  if $c>0$, it becomes a Lorentzian inner product in any case, with respect to which  $\alpha_F(x)$ is flat. Since  $\mathcal{N}(\alpha^F(x)) = {0}$ and $n\geq 4$, then $\mathcal{S}(\alpha^F(x))$ is degenerate by \eqref{fbf}. 
Thus, there exists a unit timelike vector $\zeta  \in N_{\tilde{f}}M(x)$ such that
\begin{equation}\label{zeta + eta}
i_* \zeta + \eta \in \mathcal{S}(\alpha^F) \cap \mathcal{S}(\alpha^F)^{\perp}.
\end{equation}
It follows from \eqref{eq13} and \eqref{zeta + eta}, changing $\zeta$ by $-\zeta$ if necessary, that
$
A_\zeta^{\tilde f} = \sqrt{|c|}I
$.
If $\rho \in N_{\tilde f}M(x)$ is a unit vector orthogonal to $\zeta$ in $N_{\tilde f}M(x)$, the Gauss equation of $\tilde f$
implies that 
$
\mbox{rank }  A_\rho^{\tilde f} \leq 1.
$
From now on we can argue as in the case $\epsilon=1$ to reach a contradiction.  \vspace{1ex}
\qed  

Propositions \ref{prop:n=3} and \ref{prop:n > 3} can be summarized as follows.

\begin{theorem} \label{nonexistence} There exists no hypersurface 
with nonzero constant sectional curvature in $ \mathbb{R}^k \times \mathbb{S}^{n-k+1}$ or $ \mathbb{R}^k \times \mathbb{H}^{n-k+1}$ for  $2\leq k\leq n-1$.
\end{theorem}
\section{Flat hypersurfaces of $\mathbb{R}^k \times \mathbb{Q}_\epsilon^{n-k+1}$}

To complete the local classification of hypersurfaces with constant sectional curvature $c$ of $\mathbb{R}^k \times \mathbb{S}^{n-k+1}$ and $\mathbb{R}^k \times \mathbb{H}^{n-k+1}$, $2\leq k\leq n-1$, it remains to study the case $c=0$. 
We start with the following lemmas.

\begin{lemma}\label{lemma for dimension of kernel of R}
Let $f\colon M^{n}\rightarrow \mathbb{Q}_{c_1}^{k}\times \mathbb{Q}_{c_2}^{n-k+1}$ be an isometric immersion. If $\ker R$ is invariant by $A$, then $(\ker R)^\perp$ and $(\ker (I- R))^\perp$ are totally geodesic.
\end{lemma}
\proof We argue for $(\ker R)^\perp$, the argument for $(\ker (I- R))^\perp$ being similar. 
If $X, Y\in \Gamma(\ker (I-R))$, then \eqref{derR} gives
$$\nabla_XY=R\nabla_XY+\langle AX,Y\rangle \xi\in \Gamma((\ker R)^\perp).$$
If  $X\in \Gamma(\ker (I-R)$ and $Y\in \Gamma(\ker R)$, then \eqref{derS2} yields
$$\<\nabla_X\xi, Y\>=\<(tI-R)AX, Y\>=0,$$
because  $\ker R$ is invariant by $A$. Also, using \eqref{derR} we obtain
\begin{eqnarray*}\<\nabla_\xi X, Y\>&=&\<\nabla_\xi RX, Y\>\\
&=&\<(\nabla_\xi R)X+R\nabla_X \xi, Y\>\\
&=&\<(\nabla_\xi R)X, Y\>\\
&=&\<X,\xi\>\<A\xi, Y\>+ \<A\xi, X\>\<\xi, Y\>=0.\end{eqnarray*}
Finally, using  \eqref{derS2},   for any $Y\in \Gamma(\ker R)$ we have
$$\<\nabla_\xi \xi, Y\>=\<(tI-R)A\xi, Y\>=t\<\xi, AY\>=0,$$
since $\ker R$ is invariant by $A$. Thus, $(\ker R)^\perp$ is totally geodesic.\vspace{1ex}\qed

In the next statement, and in all statements that follow, given an extrinsic product $f=\id\times g\colon \mathbb{R}^k\times M^{n-k}\to \mathbb{R}^k\times \bar{M}^{n-k+1}$, where $g\colon M^{n-k}\to  \bar{M}^{n-k+1}$ is an isometric immersion,  $\id\colon  \mathbb{R}^k \to  \mathbb{R}^k$ always denotes the identity map of the corresponding Euclidean factor.

\begin{lemma}\label{kerRsubkerA}
Let $f\colon M^{n}\rightarrow \mathbb{R}^{k}\times \mathbb{Q}_c^{n-k+1}$ be a hypersurface. If $\ker R\subset \ker A$, then one of the following possibilities holds:
\begin{itemize}
\item[(i)] If $ \dim \ker R = k$ everywhere on $M^n$, then $M^n$ is locally (isometric to) a Riemannian product  $\mathbb{R}^{k}\times M^{n-k}$, and $f$ is an extrinsic product  $f=\id\times f_2\colon \mathbb{R}^{k}\times M^{n-k}\to \mathbb{R}^{k}\times \mathbb{Q}_c^{n-k+1}$, where 
 $f_2\colon  M^{n-k}\to  \mathbb{Q}_c^{n-k+1}$ is a hypersurface;
\item[(ii)] If $ \dim \ker R = k-1$  everywhere on $M^n$, then $M^n$ is locally (isometric to) a Riemannian product $\mathbb{R}^{k-1}\times M^{n-k+1}$, and $f$ is an extrinsic product $f=\id\times f_2\colon \mathbb{R}^{k-1}\times M^{n-k+1}\to \mathbb{R}^{k-1}\times (\mathbb{R}\times \mathbb{Q}_c^{n-k+1})$,
where $f_2\colon  M^{n-k+1}\to \mathbb{R}\times \mathbb{Q}_c^{n-k+1}$ 
is a hypersurface.
 \end{itemize} 
 \end{lemma}
\proof Given $X, Y\in \Gamma(\ker R)$, by \eqref{derR} we have
$$-R\nabla_XY=(\nabla_XR)Y=\<Y, \xi\>AX+\<AX, Y\>\xi=0,$$
hence $\nabla_XY\in \Gamma(\ker R)$.  Thus, $\ker R$ is a totally geodesic distribution.   
 On the other hand, $(\ker R)^\perp$ is also totally geodesic by Lemma \ref{lemma for dimension of kernel of R}. 
If $\dim \ker R=k$, then $(i)$ holds by the second assertion in Proposition \ref{product locally splits}. 

Let $\tilde f=j\circ f$, where $j$ is the inclusion of $\mathbb{R}^{k}\times \mathbb{Q}_c^{n-k+1}$ into $\R^{n+2}$ or $\mathbb{R}_1^{n+2}$, depending on whether $c>0$ or $c< 0$, respectively.  If $\dim \ker R=k-1$, since $\ker R\subset \ker A$, then  $M^n$ is locally isometric to a Riemannian product $\mathbb{R}^{k-1}\times M^{n-k+1}$ 
 and $\tilde f$ splits as an extrinsic product $\tilde f=\id\times \tilde g$, where $\tilde g$ is the restriction of $\tilde f$ to $\{x_0\}\times M^{n-k+1}$ for some $x_0\in \mathbb{R}^{k-1}$, that is, to a leaf of $(\ker R)^\perp$ (see Proposition $7.4$ in \cite{Tojeiro}). Since $\tilde f$ takes values in $\mathbb{R}^{k}\times \mathbb{Q}_c^{n-k+1}$, then $\tilde g$ must take values in $\mathbb{R}\times \mathbb{Q}_c^{n-k+1}$, thus it gives rise to an isometric immersion  $g\colon  M^{n-k+1}\to \mathbb{R}\times \mathbb{Q}_c^{n-k+1}$ such that $\tilde g=\iota\circ g$, where  $\iota \colon \mathbb{R}\times \mathbb{Q}_c^{n-k+1} \to \mathbb{R}\times \mathbb{R}_{\sigma(c)}^{n-k+2}$ is the inclusion map.  \vspace{1ex}\qed 

We now separately consider the flat hypersurfaces of $\mathbb{R}^k \times \mathbb{S}^{n-k+1}$ and $\mathbb{R}^k \times \mathbb{H}^{n-k+1}$, $2\leq k\leq n-1$,   starting with the former case.

\begin{theorem} \label{thm:flat_sphere} Let $f\colon M^n\to \mathbb{R}^k\times \mathbb{S}^{n-k+1}$,  $2\leq k\leq n-1$, be a flat hypersurface. Then $k\in\{n-2, n-1\}$ and the following holds:
\begin{itemize}
\item[(i)] If $k=n-1$, then there exists an open and dense subset of $M^n$ where $f$ is locally   
an extrinsic product 
\begin{equation}\label{ext} f=\id \times g\colon  \mathbb{R}^{n-2}\times {M}^2\to \mathbb{R}^{n-2}\times (\mathbb{R}\times \mathbb{S}^2),\end{equation} where $g\colon M^{2}\to \mathbb{R}\times \mathbb{S}^2$ is a flat surface; if, in particular, $U\subset M^n$ is an open subset  such that $f|_U$ is as in \eqref{ext} with $g$ a vertical cylinder $g=\id \times \gamma\colon M^2=\mathbb{R}\times I\to \mathbb{R}\times \mathbb{S}^2$ over a unit speed curve $\gamma \colon I\to \mathbb{S}^2$, then $f|_U$ can also be described as an extrinsic product
\begin{equation} \label{curve}
f|_U=\id \times \gamma\colon  \mathbb{R}^{n-1}\times I\to \mathbb{R}^{n-1}\times \mathbb{S}^{2}.
\end{equation}
\item[(ii)] If $k=n-2$, then $f$ is locally an extrinsic product 
$$f=\id \times g\colon  \mathbb{R}^{n-2}\times M^2\to \mathbb{R}^{n-2}\times \mathbb{S}^{3},$$ where $g\colon M^2\to \mathbb{S}^{3}$ is a flat surface.
\end{itemize}
\end{theorem}
\proof  Let 
$
j \colon \mathbb{R}^k\times \mathbb{S}^{n-k+1}\to \mathbb{R}^k\times\mathbb{R}^{n-k+2}=\mathbb{R}^{n+2}
$ 
be the inclusion,  and let 
$
\tilde f=j \circ f \colon  \mathbb{R}^k\times \mathbb{S}^{n-k+1} \to \mathbb{R}^{n+2} 
$ 
be its composition with $f$. 
The flatness of $M^n$ implies that $\alpha^{ \tilde f}(x)$
is a flat bilinear form for any $x \in M^n$. Thus
$
\dim  \mathcal{N}(\alpha^{\tilde f}) \geq n-2
$ by \eqref{fbf}, and since $\mathcal{N}(\alpha^{\tilde f})=\ker A\cap \ker R$ by \eqref{alpha2}, 
$$
n-2\leq \dim (\ker A\cap \ker R)\leq \dim  \ker R \leq k\leq n-1
$$ 
at every $x\in M^n$. Hence $k\in\{n-2, n-1\}$, and there are four possible cases:

\begin{itemize}
\item[(a)] $\dim (\ker A\cap \ker R)= \dim  \ker R = k = n-1$;
\item[(b)] $\dim (\ker A\cap \ker R)= \dim  \ker R <k = n-1$;
\item[(c)] $\dim (\ker A\cap \ker R)= \dim  \ker R = k=n-2$;
\item[(d)] $\dim (\ker A\cap \ker R)< \dim  \ker R = k = n-1$.
\end{itemize}

Notice that if $\dim  \ker R = k = n-1$ in some open subset $U\subset M^n$, then $f|_U$ is as in \eqref{curve} by the second assertion of Proposition \ref{product locally splits}. Since $\ker R\subset \ker A$ for such an $f$, it follows that case $(d)$ cannot occur in any open subset. 

Therefore, if $k = n-1$, denoting by $U_1$ the interior of the subset of $M^n$ where case $(a)$ occurs, and by $U_2\subset M^n$ the open subset where case $(b)$ occurs, then $U_1\cup U_2$ is (open and) dense in $M^n$.
 
In $U_1$, we have  $\ker R \subset \ker A$ and $\dim \ker R= k$. Thus, $f$ is locally as in \eqref{curve} in $U_1$ by item $(i)$ of Lemma \ref{kerRsubkerA}. 

In $U_2$, we have  $\ker R \subset \ker A$ and $\dim \ker R= k-1$.
Hence,   $f$ is locally as in \eqref{ext} in $U_2$ by item $(ii)$ of Lemma \ref{kerRsubkerA}, with $M^2$ being flat by the flatness of $M^n$. This completes the proof of the statement when $k=n-1$. 

  If $k=n-2$,  case $(c)$  occurs everywhere in $M^n$, hence  $\ker R \subset \ker A$ and $\dim \ker R= k=n-2$ in $M^n$. Hence, $f$ is locally as in $(ii)$ by the item $(i)$ of Lemma \ref{kerRsubkerA}, with $M^2$ being flat by the flatness of $M^n$.
  This proves the statement for $k=n-2$ and completes the proof of the theorem.
\vspace{1ex}\qed

    Finally, we give a local classification of the flat hypersurfaces of $\mathbb{R}^k\times \mathbb{H}^{n-k+1}$, $2\leq k\leq n-1$.

\begin{theorem} \label{thm: flatH} 
If 
$
f\colon M^n\to \mathbb{R}^k\times \mathbb{H}^{n-k+1},
$ 
$
2\leq k\leq n-1,
$ 
is a flat hypersurface, then the following holds:
\begin{itemize}
\item[(i)] If $k=n-1$,  there exists an open and dense subset of $M^n$ where $f$ is locally an extrinsic product
\begin{equation}\label{h2i}
f=\id\times g\colon \mathbb{R}^{n-2}\times M^{2}\to \mathbb{R}^{n-2}\times (\mathbb{R}\times \mathbb{H}^{2}),
\end{equation}
where 
$
g\colon  M^{2}\to  \mathbb{R}\times \mathbb{H}^{2}
$ 
is a flat surface; if, in particular, $U\subset M^n$ is an open subset  such that $f|_U$ is as in \eqref{h2i} with $g$ a vertical cylinder $g=\id \times \gamma\colon M^2=\mathbb{R}\times I\to \mathbb{R}\times \mathbb{H}^2$ over a unit speed curve $\gamma \colon I\to \mathbb{H}^2$, then $f|_U$ can also be described as an extrinsic product
\begin{equation} \label{h2ii}
f|_U=\id \times \gamma\colon  \mathbb{R}^{n-1}\times I\to \mathbb{R}^{n-1}\times \mathbb{H}^{2};
\end{equation} 
\item[(ii)] If $k=n-2$, there exists an open and dense subset of $M^n$ where $f$ is locally an extrinsic product
\begin{equation}\label{h2iii}
f=\id\times h\colon \mathbb{R}^{n-3}\times M^{3}\to \mathbb{R}^{n-3}\times(\R\times  \mathbb{H}^{3}),
\end{equation}
where 
$
h\colon  M^{3}\to \mathbb{R}\times \mathbb{H}^{3}
$ 
is a flat  hypersurface; if, in particular, $U\subset M^n$ is an open subset  such that $f|_U$ is as in \eqref{h2iii} with $h$ a vertical cylinder $h=\id \times g\colon M^2=\mathbb{R}\times I\to \mathbb{R}\times \mathbb{H}^2$ over a flat surface $g \colon M^2\to \mathbb{H}^3$, then $f|_U$ can also be described as an extrinsic product
\begin{equation}\label{h2iv} 
f|_U=\id\times g\colon \mathbb{R}^{n-2}\times M^{2}\to \mathbb{R}^{n-2}\times \mathbb{H}^{3};
\end{equation} 
\item[(iii)]
If $2\leq k\leq n-3$, there exists an open and dense subset of $M^n$ where $f$ is locally an extrinsic product
\begin{equation}\label{h2v}
f=\id\times h\colon \mathbb{R}^{k-1}\times M^{n-k+1}\to \mathbb{R}^{k-1}\times(\R\times  \mathbb{H}^{n-k+1}),
\end{equation}
where 
$
h\colon  M^{n-k+1}\to \mathbb{R}\times \mathbb{H}^{n-k+1}
$ 
is a flat  hypersurface; if, in particular, $U\subset M^n$ is an open subset  such that $f|_U$ is as in \eqref{h2v} with $h$ a vertical cylinder $h=\id \times g\colon M^{n-k+1}=\mathbb{R}\times M^{n-k}\to \mathbb{R}\times \mathbb{H}^{n-k+1}$ over a horosphere $g \colon M^{n-k}\to \mathbb{H}^{n-k+1}$, then $f|_U$ can also be described as an extrinsic product
\begin{equation}\label{h2vi}
f|_U=\id\times g\colon \mathbb{R}^{k}\times M^{n-k}\to \mathbb{R}^{k}\times \mathbb{H}^{n-k+1}.
\end{equation}
\end{itemize} 
\end{theorem}

\begin{remark}
\emph{The flat hypersurfaces of $ \mathbb{R}\times \mathbb{H}^{3}$ and $ \mathbb{R}\times \mathbb{H}^{n-k+1}$ that appear
 in items $(ii)$ and $(iii)$, respectively, are all rotational hypersurfaces that were classified in Theorem 4.4 of \cite{ManfioTojeiro}.}
\end{remark}

\proof  Let 
$
j \colon \mathbb{R}^k\times \mathbb{H}^{n-k+1}\to \mathbb{R}^k\times\mathbb{R}_1^{n-k+2}=\mathbb{R}_1^{n+2}
$ 
be the inclusion,  and let 
$
\tilde f=j \circ f \colon  \mathbb{R}^k\times \mathbb{H}^{n-k+1} \to \mathbb{R}_1^{n+2} 
$ 
be its composition with $f$. 
The flatness of $M^n$ implies that $\alpha^{ \tilde f}(x)$
is a flat bilinear form for any $x \in M^n$, with respect to the Lorentzian inner product on $N_{\tilde f}M(x)$. If $\mathcal{S}(\alpha^{\tilde{f}})$ is nondegenerate, that is, $\mathcal{S}(\alpha^{\tilde{f}}) \cap \mathcal{S}(\alpha^{\tilde{f}})^{\perp} = \{0\}$, then
$
\dim  \mathcal{N}(\alpha^{\tilde f}) \geq n-2
$ by \eqref{fbf}, and since $\mathcal{N}(\alpha^{\tilde f})=\ker A\cap \ker R$ by \eqref{alpha2}, then
$$
n-2\leq \dim (\ker A\cap \ker R)\leq \dim  \ker R \leq k\leq n-1
$$ 
at every $x\in M^n$.
Hence $k\in\{n-2, n-1\}$, and we have the same four cases $(a)-(d)$ as in the proof of Theorem \ref{thm:flat_sphere}. 

As in that proof, if $
\dim \ker R = k=n-1
$ on an open subset $U\subset M^n$,  
then $f|_U$ is given as in \eqref{h2ii} by the second assertion in Proposition \ref{product locally splits}. Since $\ker R\subset \ker A$ for such a hypersurface, it follows that case $(d)$ cannot occur in any open subset.  

Assume now that $\mathcal{S}(\alpha^{\tilde{f}})$ is degenerate. Then, for a choice of one of the two unit normal vectors $N$ to $f$ in $\mathbb{R}^k \times \mathbb{H}^{n-k+1}$,  the subspace $\mathcal{S}(\alpha^{\tilde{f}})\cap \mathcal{S}(\alpha^{\tilde{f}})^{\perp}$ is spanned by the light-like vector $j_*N - f_2$.  This implies that 
\begin{eqnarray*}
    0 &=& \< \alpha^{\tilde{f}}(X,Y),  j_*N - f_2 \> \\
     &=& \< \,\,  \<AX,Y\> j_*N - \<RX,Y\>f_2, j_*N - f_2 \,\,\> \\
     &= &\<AX,Y\> - \<RX,Y\> 
\end{eqnarray*}
 for all $X,Y \in T_xM,$ hence $A=R$.

Now we prove the statement in $(iii)$. By the above, under the assumption that $2\leq k\leq n-3$,  it follows that $\mathcal{S}(\alpha^{\tilde{f}})(x)$ is degenerate for all $x\in M^n$. Denoting by $V_1$ the open subset where $\dim \ker R=k-1$ and by $V_2$ the interior of the subset where $\dim \ker R=k$, it follows that $V_1\cup V_2$ is (open and) dense in $M^n$.

  In $V_1$,  taking into account that $A=R$, it follows from part $(ii)$ of Lemma \ref{kerRsubkerA},  flatness of $M^n$, and  Theorem $7.2$  in  \cite{ManfioTojeiro}, that
  $f$ is locally  as in \eqref{h2v}. 
  
  In $V_2$, it follows from part $(i)$ of Lemma \ref{kerRsubkerA}, flatness of $M^n$, and the fact that $A=R$, that $f$ is locally as in \eqref{h2vi}. This completes the proof of the assertion in $(iii)$.

  Suppose now that $k=n-2$. Let $W_1$ be the open subset where  $\mathcal{S}(\alpha^{\tilde{f}})$ is nondegenerate, and let $W_2$ and $W_3$ be the interiors of the subsets of $M^n$ where $\mathcal{S}(\alpha^{\tilde{f}})$ is degenerate
  and $\dim \ker R$ is either $k$ or $k-1$, respectively. Then $W_1\cup W_2\cup W_3$ is (open and) dense in $M^n$. As shown in the first paragraph of the proof,  case $(c)$ occurs on $W_1$, that is, $\ker A\subset \ker R$ and $\dim \ker R=k=n-2$. Hence,  $f$ is locally as in \eqref{h2iv} in $W_1$ by item $(i)$ of Lemma \ref{kerRsubkerA}, with $M^2$ being flat by the flatness of $M^n$.

  In $W_2$, it follows from part $(i)$ of Lemma  \ref{kerRsubkerA}, flatness of $M^n$, and the fact that $A=R$, that $f$ is locally again as in \eqref{h2iv}, but now with $g$ being a horosphere.

  In $W_3$, it follows from part $(ii)$ of Lemma  \ref{kerRsubkerA},  flatness of $M^n$ and  Theorem $7.2$  in  \cite{ManfioTojeiro} that
  $f$ is locally  as in \eqref{h2iii}. The assertion in $(ii)$ is thus proved.
  
  Finally, assume that $k=n-1$. Let $U_1$ be the open subset where  $\mathcal{S}(\alpha^{\tilde{f}})$ is nondegenerate and $\dim \ker R= n-2$, let $U_2$ be the interior of the subset of $M^n$ where  $\mathcal{S}(\alpha^{\tilde{f}})$ is nondegenerate and $\dim \ker R=n-1$, and let $U_3$ and $U_4$ be the interiors of the subsets of $M^n$ where $\mathcal{S}(\alpha^{\tilde{f}})$ is degenerate
  and $\dim \ker R$ is either $k$ or $k-1$, respectively. Then  $U_1\cup U_2\cup U_3\cup U_4$ is (open and) dense in $M^n$.

  In $U_1$, case $(b)$ occurs, so  $\ker R \subset \ker A$ and $\dim \ker R= n-2=k-1$. Hence, $f$ is locally as in \eqref{h2i}  by item $(ii)$ of Lemma \ref{kerRsubkerA} ( $M^2$ being flat by the flatness of $M^n$).

  In $U_2$, case $(a)$ occurs, so  $\ker R \subset \ker A$ and $\dim \ker R= k$. Thus,  $f$ is locally as in \eqref{h2ii} by item $(i)$ of Lemma \ref{kerRsubkerA}. 

  In $U_3$, it follows from part $(i)$ of Lemma  \ref{kerRsubkerA}, flatness of $M^n$, and the fact that $A=R$, that $f$ is locally again as in \eqref{h2ii}, with $\gamma$ being a horocycle.
  
  In $U_4$, it follows from part $(ii)$ of Lemma  \ref{kerRsubkerA},  flatness of $M^n$,  and the fact that $A=R$,  that
  $f$ is locally again as in \eqref{h2i},  with $g$ being a flat rotation surface. This proves the assertion in $(i)$ and completes the proof of the theorem.\qed

\bigskip
\noindent
\textsc{Arnando N. S. Carvalho}\\
Institute of Mathematics and Computer Sciences (ICMC)\\
University of São Paulo (USP)\\
SP 13566590  São Carlos\\
Brazil|\\
\texttt{arnandonelio@usp.br}

\medskip

\noindent
\textsc{Ruy Tojeiro}\\
Institute of Mathematics and Computer Sciences (ICMC) \\
University of São Paulo (USP)\\
SP 13566590  São Carlos\\
Brazil\\
\texttt{tojeiro@icmc.usp.br}

\end{document}